\documentclass[preprint,amsmath,amssymb,aps,prl,superscriptaddress,nofootinbib,floatfix]{revtex4-2}

\usepackage{graphicx}
\usepackage{dcolumn}
\usepackage{bm}
\usepackage{mathrsfs}
\usepackage{amsmath}
\usepackage{hyperref}
\usepackage{latexsym}
\usepackage{multirow}
\usepackage{booktabs}
\usepackage{microtype}
\usepackage{algorithm}
\usepackage{algpseudocode}
\usepackage{float}
\usepackage{tabularx}
\usepackage{array}
\usepackage{xcolor}

\newcolumntype{C}{>{\centering\arraybackslash}X}   % 仅当确实要用 tabularx 时才需要

\usepackage{tikz}
\usetikzlibrary{patterns,decorations.pathmorphing,shapes}
\hypersetup{colorlinks=true, linkcolor=blue, urlcolor=blue, citecolor=magenta}

\usepackage[T1]{fontenc}
\usepackage[utf8]{inputenc}
\usepackage{lmodern,microtype}
\usepackage{amsmath,amssymb,amsfonts,amsthm,mathtools,bm}
\usepackage{booktabs,array,enumitem,xcolor,hyperref}%,geometry}
\hypersetup{colorlinks=true,linkcolor=blue!55!black,citecolor=blue!55!black,urlcolor=blue!55!black}
\allowdisplaybreaks
\newtheorem{theorem}{Theorem}[section]
\newtheorem{proposition}[theorem]{Proposition}
\newtheorem{lemma}[theorem]{Lemma}

\newtheorem{remark}[theorem]{Remark}
\newcommand{\dd}{\,\mathrm d}
\newcommand{\E}{\mathcal E}
\newcommand{\CH}{\mathrm{CH}}
\newcommand{\R}{\mathbb R}
\newcommand{\ack}[1]{\section*{Acknowledgments}#1}
\newcommand{\funding}[1]{%
  \section*{Funding}#1%
}
\newcommand{\data}[1]{%
  \section*{Data Availability Statement}#1%
}

\begin{document}

\title{The Naito--Okuda--Ouyang Biconcave Red Blood Cell Profile as a Finite-Energy Unbranched Weak Immersion with Point-Force Residues}
\author{Hao Wu}
\email{wuhao@ucas.ac.cn}
\affiliation{Zhejiang Key Laboratory of Soft Matter Biomedical Materials, Wenzhou Institute, 
University of Chinese Academy of Sciences, Wenzhou, Zhejiang 325000, China}

\date{\today}

\begin{abstract}
The axisymmetric profile \(\sin\psi(\rho)=c_0\rho\log(\rho/\rho_B)\) introduced by Naito, Okuda, and Ouyang is an exact solution of the unconstrained Helfrich shape equation for every \(\rho>0\), but its mean curvature diverges logarithmically at the dimple center. A covariant local analysis shows that in Cartesian graph coordinates the immersion is nondegenerate and belongs to \(W^{2,p}\) for every finite \(p\); it is \(C^{1,\alpha}\) for every \(\alpha<1\), but neither \(C^{1,1}\) nor \(C^2\). The conformal factor has a finite nonzero limit, so the point has branch multiplicity one and is not a branch point in the sense used in weak-immersion compactness theory. The Helfrich energy is finite, whereas the complete Euler--Lagrange operator carries the distributional residue \(\E_{\CH}=4\pi c_0\delta_p\) under the convention specified below. Hence the profile is a classical solution on the punctured surface and a finite-energy weak immersion, but it is not a free weak critical point of the unforced Canham--Helfrich functional. Stationarity is restored after adding the corresponding point-force potential or under a pinned-point constraint. Area and volume multipliers cannot cancel the Dirac residue. The apparent sign discrepancy between the shape equation used in the original papers and the modern convention is explained by the opposite definition of mean curvature. These results delimit precisely what can and cannot be inferred from modern existence and regularity theorems for Canham--Helfrich minimizers.
\end{abstract}

\maketitle

\section{Introduction}\label{sec:introduction}

The Canham--Helfrich curvature elasticity model \cite{Canham1970,Helfrich1973} has served as a central framework in membrane biophysics and soft condensed matter physics for over half a century. Early catalogues of vesicle shapes based on this model were compiled by Deuling and Helfrich \cite{Deuling1976a}. Subsequently, Seifert, Berndl, and Lipowsky \cite{Seifert1991} obtained a complete phase diagram for vesicle shape transformations within the spontaneous-curvature and bilayer-coupling models. The variational formulation of shapes of single component vesicles, developed in the early work of Svetina and \v{Z}ek\v{s} \cite{SvetinaZeks1989} and later systematized by Tu and Ou-Yang \cite{TuOuYang2004}, identifies equilibrium configurations as critical points of the bending energy subject to area and volume constraints. The first and second variations of this functional were derived in detail by Ou-Yang and Helfrich \cite{OuYangHelfrich1989}, and the resulting axisymmetric shape equation has since been studied extensively; a comprehensive review of configurations of single component fluid membranes can be found in Seifert \cite{Seifert1997}. Among the exact solutions of the spontaneous-curvature model, the prolate and oblate ellipsoids obtained by Liu et al.\ \cite{Liu1999} are particularly relevant as smooth reference shapes. Recently, a general variational framework of Helfrich free energy for multicomponent membranes has been developed for heterogeneous membrane-cortex composite layers \cite{Wu2018} and component-curvature coupling membranes \cite{Wu2025}. These developments are placed in a broader perspective in a recent review \cite{WuOuYang2026}.

Early applications of curvature elasticity to red blood cell shapes were carried out by Deuling and Helfrich \cite{Deuling1976b}. Within the framework for single component membranes, the explicit analytic biconcave solution for red blood cells discovered by Naito, Okuda, and Ouyang \cite{Naito1993,Naito1996} occupies a distinctive position. It reproduces the classical erythrocyte geometry \cite{Evans1972} and satisfies the free Helfrich shape equation at every regular point, as can be verified by direct substitution. The same equation can be derived from a stress-based point of view using the Noether identities for lipid membranes; see Capovilla and Guven \cite{CapovillaGuven2002}. However, the profile contains a logarithmic curvature divergence at each dimple center, and the question of whether such a singular configuration can be admitted as a weak solution has remained unsettled.

A rigorous assessment of a proposed shape requires evaluating its first variation on the complete admissible space, not merely substituting it into the smooth differential equation away from singular points. This distinction is decisive for the Naito--Okuda--Ouyang (NOO) biconcave profile. While the profile satisfies the free Helfrich equation on a punctured surface, a detailed geometric analysis shows that a nonzero stress flux survives around each dimple center. This flux manifests as the coefficient of a Dirac mass in the distributional shape equation, and it demands a localized point force. Related point singularities and force densities have been studied in the context of inverted catenoids and tethered membranes by Castro-Villarreal and Guven \cite{CastroGuven2007}, and force dipoles on fluid vesicles were analyzed by Guven and V\'azquez-Montejo \cite{GuvenVazquez2013}.

Furthermore, the classification of such singularities in modern geometric analysis has often been conflated with topological branch points. A finite-energy singular immersion need not be a branched immersion, and membership in a compact weak class does not imply global stationarity. The removability of point singularities for Willmore surfaces was investigated by Kuwert and Sch\"atzle \cite{KuwertSchaetzle2004}, while the weak immersion existence theory of Mondino and Scharrer \cite{MondinoScharrer2020}, together with the axisymmetric minimizer results of Choksi and Veneroni \cite{ChoksiVeneroni2013}, allows bubble trees and possibly branched limits for minimizing sequences. Lower semicontinuity properties relevant to this framework were established by Eichmann \cite{Eichmann2020}, and Li--Yau type inequalities for the Helfrich functional were obtained by Rupp and Scharrer \cite{RuppScharrer2023}. More recently, Kubin, Lussardi, and Morandotti \cite{KubinLussardiMorandotti2024} studied direct minimization on generalized Gauss graphs. These results do not automatically turn an isolated curvature singularity into a free critical point, nor do they make a nonzero Euler--Lagrange residue disappear.

The central result of this paper is the following precise classification. The logarithmic pole of the NOO profile is an unbranched, nondegenerate, finite-energy curvature singularity. The profile solves the free Helfrich equation strictly on the punctured surface. On the complete surface, it carries a quantifiable point-force residue and therefore cannot exist as an unforced free critical point.

The remainder of the paper is organized as follows. In Section~\ref{sec:methods} the geometric conventions are established, the first variation of the Canham--Helfrich functional is derived, and the axisymmetric formulation together with the exact interior calculation for the biconcave profile are presented. Section~\ref{sec:results} contains the main regularity and classification results, including the proof that the singularity is not a branch point and the distributional residue calculation. The main theorem is stated and proved. Section~\ref{sec:discussion} discusses the relation to weak immersion theory, compares with earlier parametrization analyses, and examines the physical interpretation and limitations of the profile as well as perspectives and applications. Concluding remarks are given in Section~\ref{sec:conclusion}.

\section{Methods: geometric formulation and exact calculation}\label{sec:methods}

\subsection{Conventions and first variation}\label{sec:variation}

\subsubsection{Geometric convention}
Let \(X:\Sigma\to\R^3\) be an oriented immersion with metric \(g_{ij}=\partial_iX\cdot\partial_jX\), unit normal \(n\), and second fundamental form
\[
h_{ij}=n\cdot\partial_{ij}X.
\]
The principal curvatures are \(c_1,c_2\), and
\[
2H=c_1+c_2=g^{ij}h_{ij},\qquad K=c_1c_2.
\]
With this convention, an outward-oriented sphere of radius \(R\) has \(c_1=c_2=-1/R\). The Laplace--Beltrami operator is defined by
\[
\Delta_g f=\operatorname{div}_g(\nabla_g f)=\frac1{\sqrt{\det g}}\partial_i\!\left(\sqrt{\det g}\,g^{ij}\partial_jf\right),
\]
so that in the Euclidean plane \(\Delta\log|x|=2\pi\delta_0\).

The bending energy is
\begin{equation}
\mathcal F_{\CH}[X]=\frac\kappa2\int_\Sigma(2H-c_0)^2\dd A.\label{eq:energy}
\end{equation}
The smooth Gauss--Bonnet identity is
\begin{equation}
\int_\Sigma K\dd A=2\pi\chi(\Sigma),\label{eq:GB}
\end{equation}
not \(4\pi\chi(\Sigma)\). For a sphere it gives \(4\pi\).

\subsubsection{Normal variation}
For the smooth calculation let \(X_t=X+t\eta n+o(t)\). At the nonsmooth center considered below, the same formula is obtained by working on punctured disks and using additive admissible variations \(X_t=X+tV\) with a smooth parameter-domain vector field \(V\); its scalar normal trace is \(\eta=V\cdot n\). Then
\begin{align}
\delta g_{ij}&=-2\eta h_{ij},&\delta g^{ij}&=2\eta h^{ij},\label{eq:metricvar}\\
\delta(\dd A)&=-2H\eta\dd A,&\delta(2H)&=\Delta_g\eta+|A|^2\eta.\label{eq:Hvar}
\end{align}
Indeed, \(\delta h_{ij}=\nabla_i\nabla_j\eta-\eta h_i{}^kh_{kj}\), hence
\begin{align*}
\delta(2H)&=\delta(g^{ij}h_{ij})\\
&=2\eta h^{ij}h_{ij}+g^{ij}(\nabla_i\nabla_j\eta-\eta h_i{}^kh_{kj})\\
&=\Delta_g\eta+|A|^2\eta.
\end{align*}
Writing \(u=2H-c_0\),
\begin{align}
\delta\mathcal F_{\CH}
&=\kappa\int_\Sigma u(\Delta_g\eta+|A|^2\eta)\dd A-\kappa\int_\Sigma Hu^2\eta\dd A.\label{eq:firstvar0}
\end{align}
On a smooth closed surface, integration by parts gives
\begin{equation}
\delta\mathcal F_{\CH}=\kappa\int_\Sigma\E_{\CH}\eta\dd A,\label{eq:firstvar}
\end{equation}
with
\begin{equation}
\E_{\CH}=\Delta_g(2H)+(2H-c_0)(2H^2-2K+c_0H).\label{eq:shapeoperator}
\end{equation}
The algebra is
\[
u|A|^2-Hu^2=u\{(4H^2-2K)-H(2H-c_0)\}=u(2H^2-2K+c_0H).
\]
For constrained area and oriented volume,
\[
\delta A=-\int2H\eta\dd A,\qquad\delta V=\int\eta\dd A,
\]
so the equation for \(\mathcal F_{\CH}+\lambda(A-A_0)+p(V-V_0)\) is
\begin{equation}
\kappa\E_{\CH}-2\lambda H+p=0.\label{eq:constrained}
\end{equation}
The last two terms are absolutely continuous with respect to area measure and cannot cancel an isolated Dirac mass.

\subsubsection{Why the original sign is not an error}
The original axisymmetric papers use the opposite convention
\[
H_{\rm old}=-\frac{c_1+c_2}{2}=-H.
\]
The energy \((c_1+c_2-c_0)^2\) is then \((2H_{\rm old}+c_0)^2\), and their equation
\begin{equation}
\Delta_g(2H_{\rm old})+(2H_{\rm old}+c_0)(2H_{\rm old}^2-2K-c_0H_{\rm old})=0\label{eq:oldsign}
\end{equation}
becomes minus \eqref{eq:shapeoperator} after setting \(H_{\rm old}=-H\). Multiplying by \(-1\) gives exactly \(\E_{\CH}=0\). Thus the simultaneous signs \(2H+c_0\) and \(-c_0H\) in the old convention are consistent. An earlier criticism that the spontaneous-curvature sign in the original papers is erroneous is therefore not valid; the difference arises solely from the choice of mean-curvature sign.

\subsection{Axisymmetric geometry}\label{sec:axisym}
A surface of revolution is locally
\begin{equation}
X(\rho,\phi)=(\rho\cos\phi,\rho\sin\phi,z(\rho)),\qquad z'(\rho)=\tan\psi(\rho).\label{eq:param}
\end{equation}
Here \(\psi\) is the angle between the tangent to the contour line and the positive \(\rho\)-axis. The geometry and the relevant angle are illustrated in Figure~\ref{fig:RBCcross}.
\begin{figure}[htbp]
\centering
\includegraphics[width=0.8\linewidth]{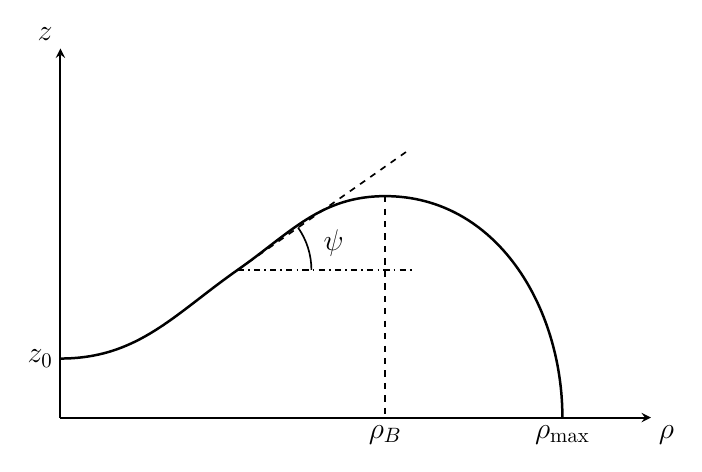}
\caption{Cross section of RBC with the rotational symmetry around the \(z\) axis and the reflection symmetry with respect to the \(\rho\) axis. \(\psi\) is the angle between the contour line and the \(\rho\) axis. Only one quadrant is shown.}
\label{fig:RBCcross}
\end{figure}
Choose \(n=(-\sin\psi\cos\phi,-\sin\psi\sin\phi,\cos\psi)\). Then the metric coefficients are
\[
g_{\rho\rho}=1+z'^2=\sec^2\psi,\qquad g_{\phi\phi}=\rho^2,\qquad g_{\rho\phi}=0.
\]
Thus the line element and area element become
\begin{align}
\dd s_g^2 =\sec^2\psi\dd\rho^2+\rho^2\dd\phi^2, \qquad \dd A =\frac\rho{\cos\psi}\dd\rho\dd\phi,\label{eq:metric}
\end{align}
where the positive branch of \(\cos\psi\) is used near the axis. The principal curvatures are computed from the parametrization. The meridional curvature is
\[
c_1=\cos\psi\,\psi',
\]
because the curvature of the planar generating curve \(z(\rho)\) is \(z''/(1+z'^2)^{3/2}\) and \(z'=\tan\psi\) gives \(z''=\sec^2\psi\,\psi'\), so the curvature is \(\cos\psi\,\psi'\). The parallel curvature is
\[
c_2=\frac{\sin\psi}{\rho}.
\]
Hence
\begin{align}
c_1 =\cos\psi\,\psi', \qquad c_2 =\frac{\sin\psi}{\rho}.\label{eq:curvaturesaxis}
\end{align}
For a radial function \(f=f(\rho)\),
\begin{equation}
\Delta_g f=\frac1{\sqrt{\det g}}\partial_\rho\left(\sqrt{\det g}\,g^{\rho\rho}\partial_\rho f\right)
=\frac{\cos\psi}{\rho}\frac\dd{\dd\rho}\left(\rho\cos\psi\,f'\right).\label{eq:lapradial}
\end{equation}
Set \(q=\sin\psi\). Then \(\cos\psi=\sqrt{1-q^2}\), and \(q'=\cos\psi\,\psi'\). The mean and Gaussian curvatures become
\begin{equation}
2H=q'+\frac q\rho,\qquad K=\frac{qq'}\rho,\qquad
\Delta_g f=(1-q^2)f''+\left(\frac{1-q^2}\rho-qq'\right)f'.\label{eq:qgeometry}
\end{equation}
The unconstrained shape equation is therefore the third-order equation
\begin{align}
0={}&(1-q^2)\left(q'''+\frac{q''}\rho-\frac{2q'}{\rho^2}+\frac{2q}{\rho^3}\right)\notag\\
&+\left(\frac{1-q^2}\rho-qq'\right)\left(q''+\frac{q'}\rho-\frac q{\rho^2}\right)\notag\\
&+\left(q'+\frac q\rho-c_0\right)\left[\frac12\left(q'+\frac q\rho\right)^2-\frac{2qq'}\rho+\frac{c_0}{2}\left(q'+\frac q\rho\right)\right].\label{eq:qODE}
\end{align}
This equation is valid only for \(\rho>0\) in a regular \(\rho\)-chart. Smooth closure at the axis is a separate geometric condition.

\subsection{Exact interior calculation for the biconcave profile}\label{sec:interior}
Let
\begin{equation}
L(\rho)=\log\frac{\rho}{\rho_B},\qquad q(\rho)=\sin\psi(\rho)=c_0\rho L(\rho).\label{eq:Naito}
\end{equation}
On every interval where \(|q|<1\),
\begin{equation}
q'=c_0(L+1),\qquad q''=\frac{c_0}{\rho},\qquad q'''=-\frac{c_0}{\rho^2}.\label{eq:qderivatives}
\end{equation}
Indeed, the first derivative is
\[
q'=c_0\left(L+\rho\frac{1}{\rho}\right)=c_0(L+1),
\]
and successive differentiations give \(q''=c_0/\rho\), \(q'''=-c_0/\rho^2\).

Substitution into \eqref{eq:qgeometry} gives
\begin{equation}
c_1=c_0(L+1),\qquad c_2=c_0L,\label{eq:principalexact}
\end{equation}
and hence
\begin{align}
2H&=c_0(2L+1),&K&=c_0^2L(L+1),&c_1-c_2&=c_0.\label{eq:HKexact}
\end{align}
The last identity shows why a \(C^2\) axis point is impossible when \(c_0\ne0\): rotational symmetry forces any smooth axis point to be umbilic.

The interior field equation follows from direct calculation. Since \((2H)'=2c_0/\rho\),
\begin{align}
\Delta_g(2H)
&=\frac{\cos\psi}{\rho}\frac\dd{\dd\rho}\left(2c_0\cos\psi\right)\notag\\
&=\frac{2c_0}{\rho}\cos\psi(\cos\psi)'\notag\\
&=-\frac{2c_0}{\rho}qq'=-2c_0^3L(L+1)=-2c_0K.\label{eq:lapH}
\end{align}
In the third equality \((\cos\psi)'=-\sin\psi\,\psi'=-q\,q'/\cos\psi\) was used, so that \(\cos\psi(\cos\psi)'=-qq'\).

Furthermore,
\begin{align}
2H^2-2K+c_0H
&=\frac{c_0^2}{2}(2L+1)^2-2c_0^2L(L+1)+\frac{c_0^2}{2}(2L+1)\notag\\
&=c_0^2(L+1),\label{eq:nonlinfactor}
\end{align}
while \(2H-c_0=2c_0L\). Thus
\begin{equation}
(2H-c_0)(2H^2-2K+c_0H)=2c_0L\cdot c_0^2(L+1)=2c_0^3L(L+1)=2c_0K.\label{eq:nonlinexact}
\end{equation}
Equations \eqref{eq:lapH} and \eqref{eq:nonlinexact} cancel exactly:
\begin{equation}
\E_{\CH}=0\qquad(\rho>0).\label{eq:punctured}
\end{equation}
This is a classical statement on the punctured surface, not yet a distributional statement on a neighborhood containing \(\rho=0\).

\subsection{Distributional Euler--Lagrange residue}\label{sec:residue}
Let \(D_\varepsilon=\{0<\rho<\varepsilon\}\). The outward unit conormal of its outer circle is \(\nu=\cos\psi\,\partial_\rho\), and \(\dd s=\varepsilon\dd\phi\). Hence
\begin{align}
\int_{\partial D_\varepsilon}\partial_\nu(2H)\dd s
&=\int_0^{2\pi}\cos\psi(\varepsilon)\frac{2c_0}{\varepsilon}\varepsilon\dd\phi\notag\\
&=4\pi c_0\cos\psi(\varepsilon)\longrightarrow4\pi c_0.\label{eq:flux}
\end{align}
This invariant flux calculation fixes the Dirac coefficient without replacing \(\Delta_g\) by a flat operator. A numerical verification of this convergence is provided in Appendix C.

\begin{proposition}[Distributional equation]\label{prop:distribution}
In a neighborhood of the dimple center, as distributions with respect to the induced area measure,
\begin{align}
\Delta_g(2H)&=-2c_0K+4\pi c_0\delta_p,\label{eq:distlap}\\
\E_{\CH}&=4\pi c_0\delta_p.\label{eq:distE}
\end{align}
\end{proposition}
\begin{proof}
Let \(\varphi\in C_c^\infty\). Integrate by parts on \(D_R\setminus D_\varepsilon\). The inner boundary term containing \(\varphi(p)\) tends to \(4\pi c_0\varphi(p)\) by \eqref{eq:flux}. The term with \(\varphi-\varphi(p)=O(\rho)\) vanishes because the boundary length is \(O(\rho)\) and \(\partial_\nu(2H)=O(1/\rho)\). The regular part is \(-2c_0K\) by \eqref{eq:lapH}. Adding the locally integrable nonlinear term \(2c_0K\) from \eqref{eq:nonlinexact} proves \eqref{eq:distE}.
\end{proof}

The singular part of the first variation is therefore
\begin{equation}
\delta\mathcal F_{\CH}[\eta]=4\pi\kappa c_0\eta(p)+\text{regular terms}.\label{eq:variationresidue}
\end{equation}
A genuinely admissible free variation detects this distribution. In the Cartesian graph chart choose \(V_0=\chi\,n(p)\), where \(\chi\in C_c^\infty\) equals one near \(p\), and use the additive family \(X_t=X+tV_0\). Since \(X\in W^{2,p}\cap C^{1,\alpha}\) and \(dX(p)\) has rank two, \(X_t\) remains a \(W^{2,p}\) immersion for small \(t\). Its normal trace \(\eta_0=V_0\cdot n\) is smooth on the punctured chart, continuous at \(p\), and satisfies \(\eta_0(p)=1\). The punctured integration-by-parts formula therefore gives the nonzero value \(4\pi\kappa c_0\). Thus the profile is not a free weak critical point. The magnitude of the required normal point force is
\begin{equation}
|F_p|=4\pi\kappa|c_0|.\label{eq:force}
\end{equation}
In ambient-vector form the local residue is \(4\pi\kappa c_0n(p)\delta_p\). Two opposite pole forces may have zero total vector sum, but each local residue is detected by a variation supported near that pole.

\subsection{Free, forced, and pinned variational problems}\label{sec:models}

\subsubsection{Point-force functional}
Add the point potential
\begin{equation}
\mathcal F_{\rm ext}[X]=-F\cdot X(p).\label{eq:pointpotential}
\end{equation}
Its normal variation is \(-F\cdot n(p)\eta(p)\). Thus the profile is stationary for the forced problem when
\begin{equation}
F\cdot n(p)=4\pi\kappa c_0,\label{eq:forcebalance}
\end{equation}
subject to the separate global closure and matching conditions.

\subsubsection{Pinned point}
If the admissible class contains \(X(p)=q\), every admissible variation satisfies \(V(p)=0\). The residue is then balanced by the reaction associated with the point constraint. Restricting scalar tests to \(\eta(p)=0\) makes the Dirac term invisible, but changes the physical problem: it describes a pinned point, not a free membrane.

\subsubsection{Area and volume constraints}
Distributionally, \eqref{eq:constrained} would read
\begin{equation}
4\pi\kappa c_0\delta_p-2\lambda H+p=0.\label{eq:cannotcancel}
\end{equation}
The last two terms are locally integrable functions times area measure; the first is singular. Uniqueness of the Lebesgue decomposition forces the singular coefficient to vanish separately. For \(c_0\ne0\), ordinary area and volume multipliers cannot restore free stationarity.

\begin{lemma}[Constraint-preserving detection of the residue]\label{lem:constraintbump}
Assume the first variations of area and volume satisfy the usual rank-two constraint qualification. More precisely, choose two smooth vector fields \(V_1,V_2\) supported in regular patches disjoint from \(p\), write \(\beta_i=V_i\cdot n\), and assume that
\begin{equation}
M=\begin{pmatrix}
-2\displaystyle\int_\Sigma H\beta_1\dd A&-2\displaystyle\int_\Sigma H\beta_2\dd A\\[2mm]
\displaystyle\int_\Sigma\beta_1\dd A&\displaystyle\int_\Sigma\beta_2\dd A
\end{pmatrix}
\label{eq:constraintmatrix}
\end{equation}
is invertible. Then there exists an admissible \(W^{2,p}\) variation curve preserving area and volume exactly, whose initial velocity has normal component equal to one at \(p\). Consequently the residue \(4\pi\kappa c_0\delta_p\) is detected even inside the doubly constrained variation space.
\end{lemma}
\begin{proof}
In the Cartesian graph chart choose \(V_0=\chi n(p)\), with \(\chi\in C_c^\infty\), \(\chi=1\) near \(p\), and support disjoint from those of \(V_1,V_2\). Consider the three-parameter additive family
\[
X_{t,a,b}=X+tV_0+aV_1+bV_2.
\]
All three vector fields are smooth in the parameter domain. Because \(X\in W^{2,p}\cap C^{1,\alpha}\) and is nondegenerate at \(p\), this family remains in the local \(W^{2,p}\) immersion class for sufficiently small parameters. Let
\[
\mathcal C(t,a,b)=\binom{A[X_{t,a,b}]-A[X]}{\operatorname{Vol}[X_{t,a,b}]-\operatorname{Vol}[X]}.
\]
The derivative of \(\mathcal C\) with respect to \((a,b)\) at the origin is exactly \(M\), because the normal traces of \(V_i\) are \(\beta_i\). The finite-dimensional implicit-function theorem therefore gives functions \(a(t),b(t)\), with \(a(0)=b(0)=0\), such that
\[
\mathcal C(t,a(t),b(t))=0
\]
for all sufficiently small positive and negative \(t\). The initial velocity is
\[
V=V_0+a'(0)V_1+b'(0)V_2.
\]
Since the correction fields vanish near \(p\),
\[
V(p)\cdot n(p)=V_0(p)\cdot n(p)=1.
\]
Its scalar normal trace is smooth away from \(p\), continuous at \(p\), and the detailed punctured integration by parts in Appendix A applies. Hence the singular part of the first variation is \(4\pi\kappa c_0\), which is nonzero when \(c_0\ne0\). The conclusion is therefore valid for an exact constraint-preserving curve, not merely for a formal scalar test function.
\end{proof}

\begin{remark}
The above lemma underscores that the impossibility of cancelling the Dirac residue is not an artifact of linearized constraints; it persists even for exact constraint-preserving variations. This provides a strong argument against the possibility that the NOO profile could be a stationary point under ordinary area and volume constraints.
\end{remark}

If the two constraint gradients fail to have rank two, the standard two-multiplier theorem itself is degenerate and must be reformulated. This exceptional situation does not create a Dirac cancellation: every ordinary geometric constraint considered here has an absolutely continuous first-variation density, whereas the residue remains singular.

\section{Results}\label{sec:results}

\subsection{Local regularity at the dimple center}\label{sec:regularity}

\subsubsection{Graph expansion and sharp regularity}
Put \(r=\rho\). From \eqref{eq:Naito},
\begin{equation}
z'(r)=\tan\psi(r)=\frac{c_0rL(r)}{\sqrt{1-c_0^2r^2L(r)^2}}.\label{eq:zprime}
\end{equation}
Since \(rL(r)\to0\),
\[
z'(r)=c_0rL(r)+O(r^3|L(r)|^3).
\]
Using
\[
\int r\log(r/\rho_B)\dd r=\frac{r^2}{2}\log(r/\rho_B)-\frac{r^2}{4},
\]
we obtain
\begin{equation}
z(r)-z(0)=\frac{c_0}{2}r^2L(r)-\frac{c_0}{4}r^2+O(r^4|L(r)|^3).\label{eq:graphasymptotic}
\end{equation}
Thus the surface is the Cartesian graph
\begin{equation}
X(x,y)=(x,y,u(x,y)),\qquad u(x,y)=z(\sqrt{x^2+y^2}).\label{eq:graph}
\end{equation}
For a radial function, the Hessian eigenvalues are \(u_{rr}\) and \(u_r/r\). With \(q=c_0rL\),
\begin{align}
\frac{u_r}{r}&=\frac{c_0L}{\sqrt{1-q^2}},\label{eq:hess1}\\
u_{rr}&=\frac{q'}{(1-q^2)^{3/2}}=\frac{c_0(L+1)}{(1-q^2)^{3/2}}.\label{eq:hess2}
\end{align}
Therefore \(|D^2u|\le C(1+|\log r|)\). Since
\begin{equation}
\int_0^\varepsilon r|\log r|^p\dd r<\infty\qquad(p<\infty),\label{eq:logintegrable}
\end{equation}
it follows that \(u\in W^{2,p}\) for every finite \(p\). Also \(u_r=O(r|\log r|)\to0\), so
\begin{equation}
\partial_xX(0)=(1,0,0),\qquad\partial_yX(0)=(0,1,0).\label{eq:rank}
\end{equation}
The differential has rank two. By the two-dimensional Sobolev embedding,
\begin{equation}
X\in C^{1,\alpha}\quad\text{for every }\alpha<1.\label{eq:C1a}
\end{equation}
Because \eqref{eq:hess1}--\eqref{eq:hess2} are unbounded for \(c_0\ne0\), the graph is neither \(C^{1,1}\) nor \(C^2\).

\subsubsection{Finite curvature and finite Helfrich energy}
The exact area element is
\[
\dd A=\frac{r}{\sqrt{1-c_0^2r^2L^2}}\dd r\dd\phi.
\]
Using \eqref{eq:principalexact},
\begin{align}
\int_{D_\varepsilon}|A|^2\dd A
&=2\pi c_0^2\int_0^\varepsilon\frac{r\{(L+1)^2+L^2\}}{\sqrt{1-c_0^2r^2L^2}}\dd r<\infty,\label{eq:Atotal}\\
\mathcal F_{\CH}(D_\varepsilon)
&=4\pi\kappa c_0^2\int_0^\varepsilon\frac{rL^2}{\sqrt{1-c_0^2r^2L^2}}\dd r<\infty.\label{eq:energylocal}
\end{align}
The leading asymptotic is
\begin{equation}
\mathcal F_{\CH}(D_\varepsilon)=2\pi\kappa c_0^2\varepsilon^2\left(L_\varepsilon^2-L_\varepsilon+\frac12\right)+O(\varepsilon^4|L_\varepsilon|^4),\label{eq:energyasymptotic}
\end{equation}
where \(L_\varepsilon=\log(\varepsilon/\rho_B)\). This energy tends to zero even though the curvature diverges.

\subsection{The point is not a branch point}\label{sec:nobranch}
The rank computation \eqref{eq:rank} already excludes a geometric branch point. An independent conformal calculation exposes the source of the contrary claim.

Let \(r\) now denote an isothermal radius and require
\begin{equation}
\sec^2\psi\dd\rho^2+\rho^2\dd\phi^2=e^{2\lambda(r)}(\dd r^2+r^2\dd\phi^2).\label{eq:isothermal}
\end{equation}
Comparing coefficients gives
\begin{equation}
e^\lambda r=\rho,\qquad e^\lambda\dd r=\sec\psi\dd\rho,\qquad\frac{\dd r}{r}=\frac{\dd\rho}{\rho\cos\psi}.\label{eq:riso}
\end{equation}
Since \(q=c_0\rho L\),
\begin{equation}
\frac1{\cos\psi}=\frac1{\sqrt{1-q^2}}=1+O(\rho^2L^2).\label{eq:cosinv}
\end{equation}
Consequently,
\begin{align}
\log\frac r\rho&=\log C+\int_0^\rho\frac{1/\cos\psi(t)-1}{t}\dd t=\log C+O(\rho^2L^2),\label{eq:rhoratio}
\end{align}
and
\begin{equation}
r=C\rho\{1+O(\rho^2L^2)\},\qquad e^\lambda=\frac\rho r\longrightarrow C^{-1}>0.\label{eq:lambdabounded}
\end{equation}
For a weak conformal immersion with branch multiplicity \(n\), the standard local form is
\begin{equation}
\lambda(z)=(n-1)\log|z|+\omega(z),\label{eq:branchform}
\end{equation}
with \(\omega\) bounded under the usual hypotheses \cite{MondinoScharrer2020,Riviere2008,BernardRiviere2013}. Equation \eqref{eq:lambdabounded} gives \(n=1\), not \(n=2\).

If \(t=\log\rho\), then even the flat polar metric is
\[
\dd\rho^2+\rho^2\dd\phi^2=e^{2t}(\dd t^2+\dd\phi^2).
\]
The factor \(e^t\) is the ordinary polar Jacobian, not a branch factor. Omitting it assigns a fictitious branch to a flat disk. Similarly, the normal height term \(r^2\log r\) cannot determine branch order because the tangential components of \(X(x,y)\) already have nonzero linear terms.

\subsection{Main theorem}\label{sec:maintheorem}
\begin{theorem}[Classification of the logarithmic dimple]\label{thm:main}
Let \(c_0\ne0\), \(\rho_B>0\), and let a surface of revolution near \(\rho=0\) satisfy
\[
\sin\psi(\rho)=c_0\rho\log(\rho/\rho_B),\qquad z'(\rho)=\tan\psi(\rho),
\]
on a sufficiently small interval where the square root in \eqref{eq:zprime} is real. Then the following statements hold. The surface extends across \(\rho=0\) as a nondegenerate graph immersion in \(W^{2,p}_{\rm loc}\) for every finite \(p\), hence in \(C^{1,\alpha}\) for every \(\alpha<1\). It is neither \(C^{1,1}\) nor \(C^2\), and its principal curvatures diverge logarithmically. Its total curvature and Canham--Helfrich energy are finite near the point. Its branch multiplicity is one: the differential has rank two and the conformal factor tends to a positive finite limit. It satisfies \(\E_{\CH}=0\) classically on the punctured neighborhood. On the complete neighborhood, \(\E_{\CH}=4\pi c_0\delta_p\) in distributions. It is not a critical point of the unforced free functional under arbitrary compactly supported normal variations, and ordinary area/volume multipliers do not cancel the residue. It is stationary for a point-forced problem with normal force coefficient \(4\pi\kappa c_0\), or for a pinned-point problem whose reaction supplies that coefficient.
\end{theorem}
\begin{proof}
The first three statements follow from \eqref{eq:graphasymptotic}--\eqref{eq:energylocal}. The branch multiplicity statement follows independently from \eqref{eq:rank} and \eqref{eq:lambdabounded}. The classical punctured equation is \eqref{eq:punctured}. The distributional equation is Proposition \ref{prop:distribution}. The failure of free stationarity follows from \eqref{eq:variationresidue} and the singular/absolutely-continuous decomposition in \eqref{eq:cannotcancel}. The stationarity under point forcing or pinning follows from \eqref{eq:pointpotential}--\eqref{eq:forcebalance} and the pinned-point constraint.
\end{proof}

\subsection{Relation to weak immersion and minimization theory}\label{sec:weak}
The weak-immersion framework is indispensable for compactness: minimizing sequences can develop concentration, multiple covers, necks, and bubble trees. Nevertheless, the valid implication is
\[
\text{minimizer}\Longrightarrow\text{stationary}\Longrightarrow\text{admissible finite-energy object},
\]
where the first arrow presupposes an appropriate constraint qualification. Neither converse is valid in general, and finite energy alone does not verify every technical requirement of a chosen admissible class.

Mondino and Scharrer prove existence in a bubble-tree compactification and regularity away from genuine branch points \cite{MondinoScharrer2020}. A local weak equation tested only on the complement of a marked set proves regularity and stationarity there; it does not establish global free stationarity at the omitted points. Here the issue is sharper because Sections \ref{sec:regularity} and \ref{sec:nobranch} prove that the dimple center is an ordinary, nondegenerate immersion point of multiplicity one. It cannot be removed from the global variation space by relabeling it as a branch point.

For a genuine branched conformal immersion with local multiplicities \(n_j\), the generalized identity is
\begin{equation}
\int_\Sigma K\dd\mu=2\pi\chi(\Sigma)+2\pi\sum_j(n_j-1),\label{eq:branchedGB}
\end{equation}
under the standard closed-surface hypotheses. Here the local multiplicity is one, so there is no branch correction.

Bubble-tree examples show that a generalized minimizer may consist of smooth components joined only in the limiting description. They do not identify the present logarithmic graph with a collapsed catenoidal neck. Such an identification would require an explicit family of smooth critical or minimizing surfaces, a specified convergence topology, area/volume and energy control, and a proof matching the limiting neck stress to \eqref{eq:force}. The term \(r^2\log r\) alone proves none of these statements. Therefore, the interpretation as a catenoidal neck collapse remains an open problem and is not a conclusion of the present work.

\section{Discussion}\label{sec:discussion}

\subsection{Comparison with the parametrization analysis}\label{sec:1995}
Podgornik, Svetina, and \v{Z}ek\v{s} showed that standard parametrizations yield equivalent Euler--Lagrange equations for regular smooth spherical contours, while nonanalytic higher derivatives at poles correspond to altered physical problems involving point sources or geometric constraints \cite{Podgornik1995}. For the constraints in their numerical example, the zero-force smooth branch had lower bending energy than the forced branches. The present calculation is consistent with that analysis and adds the exact covariant Dirac coefficient \(4\pi\kappa c_0\), the sharp local regularity \(W^{2,p}\) for all finite \(p\) but not \(C^2\), two independent proofs that the singularity is unbranched, and the residue in the full nonlinear shape equation rather than only in its linearization. The point-force interpretation itself is therefore not wholly new; the present contribution is the precise classification, covariant residue, and correction of the branch-point claim. In particular, the result provides a rigorous, nonlinear justification for the physical intuition in \cite{Podgornik1995} that nonanalytic pole behaviour signals external forces, and it quantifies the force coefficient explicitly.

\subsection{Physical interpretation and limitations}\label{sec:physical}
The theorem supports a specific but limited conclusion. The formula is an exact interior solution and a legitimate finite-energy weak geometric object. The curvature divergence is sufficiently mild that the local bending energy tends to zero. These facts explain why it can approximate a biconcave contour over most of the surface.

The assertion that the singularity is automatically permitted as a branch point of a free Canham--Helfrich minimizer is not supported. It is not a branch point, and its full first variation contains a nonzero point residue. Accordingly, the profile represents a membrane subject to localized normal forces, a membrane with pinned dimple centers or equivalent point constraints, a punctured-domain outer solution to be matched to a small regularized core, or a possible limit of a specially constructed constrained sequence whose convergence is proved separately. It is not established as an unforced free minimizer.

The sharp \(W^{2,p}\) regularity (all finite \(p\), but not \(C^2\)) implies that the square of the mean curvature is integrable, which accounts for the finite bending energy despite the logarithmic divergence of \(H\). At the same time, the non-integrability of the second derivatives in \(L^\infty\) leads to a nonzero stress flux, manifesting as the point residue. This dichotomy illustrates how finite energy and localized forces can coexist in a weak geometric setting.

For a real red blood cell, spectrin shear elasticity, bilayer area difference, heterogeneity, active stress, and global constraints may regularize or support the dimple. The ideal fluid-membrane profile must therefore be presented as a singular reduced model, not as a complete erythrocyte mechanics model.

\subsection{Perspectives and applications}\label{sec:impacts}
Beyond resolving a longstanding ambiguity in membrane geometry, this classification bridges fundamental mathematical physics with the mechanics of complex and non-equilibrium soft matter. In biophysical modeling, it rigorously proves that a purely fluid Canham--Helfrich membrane cannot sustain this exact biconcave geometry spontaneously. Physical erythrocytes must rely on additional symmetry-breaking mechanisms, such as the shear elasticity of the spectrin cytoskeleton \cite{Discher1994,Lim2002,Mukhopadhyay2002}, localized protein pinning, or spontaneous curvature heterogeneities, to balance this theoretical force residue.

More broadly, this singular framework also offers a template for studying geometric constraints in active matter hydrodynamics \cite{Wu2015,Wu2016,Marchetti2013,SalbreuxJulicher2017,Farutin2019}. Just as active nematic films exhibit self-driven flows governed by the dynamics of discrete topological defects, active biological membranes may generate localized, non-reciprocal stress fields capable of dynamically sustaining point-force singularities. Formulating generalized Helfrich equations with localized active stress gradients presents a natural extension of the distributional calculus developed here.

Finally, the methodology of extracting discrete force residues from continuous shape equations impacts the broader study of two-dimensional fluctuating membranes and mechanical metamaterials. Recognizing that finite bending energy can mask underlying localized forces ensures that future theoretical evaluations of minimal surfaces, whether in fluid bilayers, non-local electro-elastic media, or odd-elastic membranes \cite{Scheibner2020}, will structurally account for hidden distributional sources.

\section{Conclusion}\label{sec:conclusion}
Modern weak immersion theory provides an indispensable framework for understanding why finite-energy limits, singular necks, and bubble trees must be admitted in variational geometry \cite{MondinoScharrer2020}. However, it does not universally convert every isolated curvature singularity into a free equilibrium state. For the NOO profile \cite{Naito1993,Naito1996}, the covariant analysis establishes that the metric is nondegenerate, the conformal factor remains strictly positive and bounded, and the geometric branch multiplicity is exactly one. Because the stress flux limits to a distributional point residue of \(4\pi\kappa c_0\delta_p\), the strongest conclusion supported by both the explicit formula and weak existence theory is that this profile is an unbranched finite-energy weak immersion with point-force residues, classically unforced only on the punctured surface.

\appendix
\section*{Appendix}
\subsection*{A. Detailed distributional integration by parts}
Let \(u=2H\), \(A_{\varepsilon,R}=D_R\setminus\overline{D_\varepsilon}\), and \(\varphi\in C_c^\infty(D_R)\). Green's identity gives
\begin{align}
\int_{A_{\varepsilon,R}}u\Delta_g\varphi\dd A
=\int_{A_{\varepsilon,R}}\varphi\Delta_g u\dd A+\int_{\partial A_{\varepsilon,R}}(u\partial_\nu\varphi-\varphi\partial_\nu u)\dd s.
\end{align}
The outer contribution vanishes when the support lies inside \(D_R\). The inner conormal of the annulus points toward decreasing \(\rho\), hence
\begin{align*}
\int_{\partial D_\varepsilon}(u\partial_{\nu_A}\varphi-\varphi\partial_{\nu_A}u)\dd s
=-\int_{\partial D_\varepsilon}u\partial_{\nu_D}\varphi\dd s+\int_{\partial D_\varepsilon}\varphi\partial_{\nu_D}u\dd s.
\end{align*}
The first term tends to zero because \(u=O(|\log\varepsilon|)\) and the boundary length is \(O(\varepsilon)\). Split the second as
\[
\varphi(p)\int_{\partial D_\varepsilon}\partial_{\nu_D}u\dd s+\int_{\partial D_\varepsilon}(\varphi-\varphi(p))\partial_{\nu_D}u\dd s.
\]
The last integral vanishes because \(\varphi-\varphi(p)=O(\varepsilon)\), \(\partial_{\nu_D}u=O(1/\varepsilon)\), and \(\dd s=O(\varepsilon)\). The first tends to \(4\pi c_0\varphi(p)\), proving \eqref{eq:distlap} with all orientations accounted for.

\subsection*{B. Absence of an intrinsic cone defect}
From \eqref{eq:metric},
\[
g_{\rho\rho}=1+c_0^2\rho^2L^2+O(\rho^4L^4),\qquad g_{\phi\phi}=\rho^2.
\]
The geodesic radius is \(s(\rho)=\rho+O(\rho^3L^2)\), while the circumference is exactly \(2\pi\rho\). Hence
\[
\frac{\operatorname{Length}(\partial D_\rho)}{s(\rho)}=2\pi+O(\rho^2L^2)\longrightarrow2\pi.
\]
There is no cone-angle deficit. The Dirac mass belongs to the fourth-order shape operator, not to the intrinsic Gaussian-curvature measure.

\subsection*{C. Numerical convergence indicators}
For \(c_0=-1\), \(\rho_B=1\), the exact flux is \(-4\pi\sqrt{1-\varepsilon^2\log^2\varepsilon}\).
\begin{table*}[b]
\centering
\renewcommand{\arraystretch}{1.5}
\setlength{\tabcolsep}{14pt}
\begin{tabular*}{\textwidth}{@{\extracolsep{\fill}}rrrr@{}}
\toprule
\(\varepsilon\) & \(q(\varepsilon)\) & flux & relative error \\
\midrule
0.1   & 0.230259 & \(-12.228706\) & \(2.6871\times10^{-2}\) \\
0.05  & 0.149787 & \(-12.424601\) & \(1.1282\times10^{-2}\) \\
0.02  & 0.078240 & \(-12.527849\) & \(3.0655\times10^{-3}\) \\
0.01  & 0.046052 & \(-12.553038\) & \(1.0609\times10^{-3}\) \\
0.005 & 0.026492 & \(-12.561960\) & \(3.5096\times10^{-4}\) \\
0.001 & 0.006908 & \(-12.566071\) & \(2.3859\times10^{-5}\) \\
\bottomrule
\end{tabular*}
\end{table*}
%
% Each of the commands below will create an unnumbered section with the appropriate heading.
% Remove any sections that are not relevant for your article.
% All sections except suppdata will be removed if the [anonymous] option is used.
% See iopjournal-guidelines.pdf for more information.
%

\ack{H.W. thanks R. Ma for informative discussions.}

\funding{H.W. is supported by the General Program of National Natural Science Foundation of China under Grant No. 12374210 and the open research fund of Songshan Lake Materials Laboratory No. 2023SLABFN20.}
% This section is a list of funder names and grant numbers

%\roles{Sample text inserted for demonstration.}
% List author names and the contributions made to the article, using terms from the NISO Contributor Roles Taxonomy (CRediT) https://credit.niso.org

\data{No new data were created or analysed in this study.}%All data that support the findings of this study are included within the article (and any supplementary files).}%No new data were created or analysed in this study.}%All data that support the findings of this study are included within the article.}
% For more information on IOP Publishing's research data policy see: https://publishingsupport.iopscience.iop.org/questions/research-data/

%\suppdata{Sample text inserted for demonstration.}

%\section*{References}


\begin{thebibliography}{99}
\bibitem{Canham1970} P. B. Canham, ``The minimum energy of bending as a possible explanation of the biconcave shape of the human red blood cell,'' \emph{J. Theor. Biol.} \textbf{26}, 61--81 (1970), \href{https://doi.org/10.1016/S0022-5193(70)80032-7}{doi:10.1016/S0022-5193(70)80032-7}.
\bibitem{Helfrich1973} W. Helfrich, ``Elastic properties of lipid bilayers: theory and possible experiments,'' \emph{Z. Naturforsch. C} \textbf{28}, 693--703 (1973), \href{https://doi.org/10.1515/znc-1973-11-1209}{doi:10.1515/znc-1973-11-1209}.
\bibitem{Deuling1976a} H. J. Deuling and W. Helfrich, ``The curvature elasticity of fluid membranes: A catalogue of vesicle shapes,'' \emph{J. Phys.} \textbf{37}, 1335--1345 (1976), \href{https://doi.org/10.1051/jphys:0197600370110133500}{doi:10.1051/jphys:0197600370110133500}.
\bibitem{Seifert1991} U. Seifert, K. Berndl, and R. Lipowsky, ``Shape transformations of vesicles: Phase diagram for spontaneous-curvature and bilayer-coupling models,'' \emph{Physical Review A} \textbf{44}, 1182 (1991), \href{https://doi.org/10.1103/PhysRevA.44.1182}{doi:10.1103/PhysRevA.44.1182}.
\bibitem{SvetinaZeks1989} S. Svetina and B. \v{Z}ek\v{s}, ``Membrane bending energy and shape determination of phospholipid vesicles and red blood cells,'' \emph{Eur. Biophys. J.} \textbf{17}, 101--111 (1989), \href{https://doi.org/10.1007/BF00257107}{doi:10.1007/BF00257107}.
\bibitem{TuOuYang2004} Z. Tu and Z. Ouyang, ``A geometric theory on the elasticity of bio-membranes,'' \emph{J. Phys. A: Math. Gen.} \textbf{37}, 11407--11429 (2004), \href{https://doi.org/10.1088/0305-4470/37/47/009}{doi:10.1088/0305-4470/37/47/009}.
\bibitem{OuYangHelfrich1989} Z. Ouyang and W. Helfrich, ``Bending energy of vesicle membranes: General expressions for the first, second, and third variation of the shape energy and applications to spheres and cylinders,'' \emph{Phys. Rev. A} \textbf{39}, 5280--5288 (1989), \href{https://doi.org/10.1103/PhysRevA.39.5280}{doi:10.1103/PhysRevA.39.5280}.
\bibitem{Seifert1997} U. Seifert, ``Configurations of fluid membranes and vesicles,'' \emph{Adv. Phys.} \textbf{46}, 13--137 (1997), \href{https://doi.org/10.1080/00018739700101488}{doi:10.1080/00018739700101488}.
\bibitem{Liu1999} Q. Liu, H. Zhou, J. Liu, and Z. Ouyang, ``Spheres and prolate and oblate ellipsoids from an analytical solution of the spontaneous-curvature fluid-membrane model,'' \emph{Phys. Rev. E} \textbf{60}, 3227--3233 (1999), \href{https://doi.org/10.1103/PhysRevE.60.3227}{doi:10.1103/PhysRevE.60.3227}.
\bibitem{Wu2018} H. Wu, M. A. Ponce de Le\'{o}n, and H. G. Othmer, ``Getting in shape and swimming: The role of cortical forces and membrane heterogeneity in eukaryotic cells,'' \emph{J. Math. Biol.} \textbf{77}, 595--626 (2018), \href{https://doi.org/10.1007/s00285-018-1218-3}{doi:10.1007/s00285-018-1218-3}.
\bibitem{Wu2025} H. Wu and Z. Ouyang, ``A generalized Helfrich free energy framework for multicomponent fluid membranes,'' \emph{Membranes} \textbf{15}, 182 (2025), \href{https://doi.org/10.3390/membranes15060182}{doi:10.3390/membranes15060182}.
\bibitem{WuOuYang2026} H. Wu and Z. Ouyang, ``Unveiling the Heterogeneity and Multifunctions of Biological Membranes: A Unifying Perspective in Membrane Biophysics from Molecular Coupling to Cellular Function,'' \emph{Membranes} \textbf{16}, 79 (2026), \href{https://doi.org/10.3390/membranes16030079}{doi:10.3390/membranes16030079}.
\bibitem{Deuling1976b} H. J. Deuling and W. Helfrich, ``Red blood cell shapes as explained on the basis of curvature elasticity,'' \emph{Biophys. J.} \textbf{16}, 861--868 (1976), \href{https://doi.org/10.1016/S0006-3495(76)85736-0}{doi:10.1016/S0006-3495(76)85736-0}.
\bibitem{Naito1993} H. Naito, M. Okuda, and Z. Ouyang, ``Counterexample to some shape equations for axisymmetric vesicles,'' \emph{Phys. Rev. E} \textbf{48}, 2304--2307 (1993), \href{https://doi.org/10.1103/PhysRevE.48.2304}{doi:10.1103/PhysRevE.48.2304}.
\bibitem{Naito1996} H. Naito, M. Okuda, and Z. Ouyang, ``Polygonal shape transformation of a circular biconcave vesicle induced by osmotic pressure,'' \emph{Phys. Rev. E} \textbf{54}, 2816--2826 (1996), \href{https://doi.org/10.1103/PhysRevE.54.2816}{doi:10.1103/PhysRevE.54.2816}.
\bibitem{Evans1972} E. Evans and Y. C. Fung, ``Improved measurements of the erythrocyte geometry,'' \emph{Microvasc. Res.} \textbf{4}, 335--347 (1972), \href{https://doi.org/10.1016/0026-2862(72)90069-6}{doi:10.1016/0026-2862(72)90069-6}.
\bibitem{CapovillaGuven2002} R. Capovilla and J. Guven, ``Stresses in lipid membranes,'' \emph{J. Phys. A: Math. Gen.} \textbf{35}, 6233--6247 (2002), \href{https://doi.org/10.1088/0305-4470/35/30/302}{doi:10.1088/0305-4470/35/30/302}.
\bibitem{CastroGuven2007} P. Castro-Villarreal and J. Guven, ``Inverted catenoids, curvature singularities, and tethered membranes,'' \emph{Phys. Rev. E} \textbf{76}, 011922 (2007), \href{https://doi.org/10.1103/PhysRevE.76.011922}{doi:10.1103/PhysRevE.76.011922}.
\bibitem{GuvenVazquez2013} J. Guven and P. V\'azquez-Montejo, ``Force dipoles and stable local defects on fluid vesicles,'' \emph{Phys. Rev. E} \textbf{87}, 042710 (2013), \href{https://doi.org/10.1103/PhysRevE.87.042710}{doi:10.1103/PhysRevE.87.042710}.
\bibitem{KuwertSchaetzle2004} E. Kuwert and R. Sch\"atzle, ``Removability of point singularities of Willmore surfaces,'' \emph{Ann. of Math.} \textbf{160}, 315--357 (2004), \href{https://doi.org/10.4007/annals.2004.160.315}{doi:10.4007/annals.2004.160.315}.
\bibitem{MondinoScharrer2020} A. Mondino and C. Scharrer, ``Existence and regularity of spheres minimising the Canham--Helfrich energy,'' \emph{Arch. Ration. Mech. Anal.} \textbf{236}, 1455--1485 (2020), \href{https://doi.org/10.1007/s00205-020-01497-4}{doi:10.1007/s00205-020-01497-4}.
\bibitem{ChoksiVeneroni2013} R. Choksi and M. Veneroni, ``Global minimizers for the doubly-constrained Helfrich energy: The axisymmetric case,'' \emph{Calc. Var. Partial Differential Equations} \textbf{48}, 337--366 (2013), \href{https://doi.org/10.1007/s00526-012-0553-9}{doi:10.1007/s00526-012-0553-9}.
\bibitem{Eichmann2020} S. Eichmann, ``Lower semicontinuity for the Helfrich problem,'' \emph{Ann. Global Anal. Geom.} \textbf{58}, 147--175 (2020), \href{https://doi.org/10.1007/s10455-020-09718-5}{doi:10.1007/s10455-020-09718-5}.
\bibitem{RuppScharrer2023} F. Rupp and C. Scharrer, ``Li--Yau inequalities for the Helfrich functional and applications,'' \emph{Calc. Var. Partial Differential Equations} \textbf{62}, Article 45 (2023), \href{https://doi.org/10.1007/s00526-022-02381-7}{doi:10.1007/s00526-022-02381-7}.
\bibitem{KubinLussardiMorandotti2024} A. Kubin, L. Lussardi, and M. Morandotti, ``Direct minimization of the Canham--Helfrich energy on generalized Gauss graphs,'' \emph{J. Geom. Anal.} \textbf{34}, Article 121 (2024), \href{https://doi.org/10.1007/s12220-024-01564-2}{doi:10.1007/s12220-024-01564-2}.
\bibitem{Riviere2008} T. Rivi\`ere, ``Analysis aspects of Willmore surfaces,'' \emph{Invent. Math.} \textbf{174}, 1--45 (2008), \href{https://doi.org/10.1007/s00222-008-0129-7}{doi:10.1007/s00222-008-0129-7}.
\bibitem{BernardRiviere2013} Y. Bernard and T. Rivi\`ere, ``Singularity removability at branch points for Willmore surfaces,'' \emph{Pacific J. Math.} \textbf{265}, 257--311 (2013), \href{https://doi.org/10.2140/pjm.2013.265.257}{doi:10.2140/pjm.2013.265.257}.
\bibitem{Podgornik1995} R. Podgornik, S. Svetina, and B. \v{Z}ek\v{s}, ``Parametrization invariance and shape equations of elastic axisymmetric vesicles,'' \emph{Phys. Rev. E} \textbf{51}, 544--547 (1995), \href{https://doi.org/10.1103/PhysRevE.51.544}{doi:10.1103/PhysRevE.51.544}.
\bibitem{Discher1994} D. E. Discher, N. Mohandas, and E. A. Evans, ``Molecular Maps of Red Cell Deformation: Hidden Elasticity and in Situ Connectivity,'' \emph{Science} \textbf{266}, 1032--1035 (1994), \href{https://doi.org/10.1126/science.7973655}{doi:10.1126/science.7973655}.
\bibitem{Lim2002} H. W. G. Lim, M. Wortis, and R. Mukhopadhyay, ``Stomatocyte--discocyte--echinocyte sequence of the human red blood cell: Evidence for the bilayer--couple hypothesis from membrane mechanics,'' \emph{Proc. Natl. Acad. Sci. U.S.A.} \textbf{99}, 16766--16769 (2002), \href{https://doi.org/10.1073/pnas.202617299}{doi:10.1073/pnas.202617299}.
\bibitem{Mukhopadhyay2002} R. Mukhopadhyay, H. W. G. Lim, and M. Wortis, ``Echinocyte shapes: Bending, stretching, and shear determine spicule shape and spacing,'' \emph{Biophys. J.} \textbf{82}, 1756--1772 (2002), \href{https://doi.org/10.1016/S0006-3495(02)75527-6}{doi:10.1016/S0006-3495(02)75527-6}.
\bibitem{Wu2015} H. Wu, M. Thi\'{e}baud, W.-F. Hu, A. Farutin, S. Rafa\"{\i}, M.-C. Lai, P. Peyla, and C. Misbah, ``Amoeboid motion in confined geometry,'' \emph{Phys. Rev. E} \textbf{92}, 050701(R) (2015), \href{https://doi.org/10.1103/PhysRevE.92.050701}{doi:10.1103/PhysRevE.92.050701}.
\bibitem{Wu2016} H. Wu, A. Farutin, W.-F. Hu, M. Thi\'{e}baud, S. Rafa\"{\i}, P. Peyla, M.-C. Lai, and C. Misbah, ``Amoeboid swimming in a channel,'' \emph{Soft Matter} \textbf{12}, 7470--7484 (2016), \href{https://doi.org/10.1039/C6SM01235F}{doi:10.1039/C6SM01235F}.
\bibitem{Marchetti2013} M. C. Marchetti, J. F. Joanny, S. Ramaswamy, T. B. Liverpool, J. Prost, M. Rao, and R. A. Simha, ``Hydrodynamics of soft active matter,'' \emph{Rev. Mod. Phys.} \textbf{85}, 1143--1189 (2013), \href{https://doi.org/10.1103/RevModPhys.85.1143}{doi:10.1103/RevModPhys.85.1143}.
\bibitem{SalbreuxJulicher2017} G. Salbreux and F. J\"ulicher, ``Mechanics of active surfaces,'' \emph{Phys. Rev. E} \textbf{96}, 032404 (2017), \href{https://doi.org/10.1103/PhysRevE.96.032404}{doi:10.1103/PhysRevE.96.032404}.
\bibitem{Farutin2019} A. Farutin, H. Wu, W.-F. Hu, S. Rafa\"{\i}, P. Peyla, M.-C. Lai, and C. Misbah, ``Analytical study for swimmers in a channel,'' \emph{J. Fluid Mech.} \textbf{881}, 365--383 (2019), \href{https://doi.org/10.1017/jfm.2019.782}{doi:10.1017/jfm.2019.782}.
\bibitem{Scheibner2020} C. Scheibner, A. Souslov, A. Murugan, C. Riedel, M. Kruse, and V. Vitelli, ``Odd elasticity,'' \emph{Nat. Phys.} \textbf{16}, 475--480 (2020), \href{https://doi.org/10.1038/s41567-020-0795-y}{doi:10.1038/s41567-020-0795-y}.
\end{thebibliography}
\end{document}